\documentclass[12pt]{amsart}
\usepackage{amssymb,amsmath,amsthm,mathrsfs}
\usepackage{graphicx}
\newtheorem{thm}{Theorem}
\newtheorem{theorem}[thm]{Theorem}

\newtheorem{corollary}[thm]{Corollary}
\newtheorem{proposition}[thm]{Proposition}
\numberwithin{equation}{section} \theoremstyle{definition}

\newcommand{\R}{\mathbb{R}}

\newcommand{\Z}{\mathbb{Z}}

\newcommand{\Cayley}{\mathcal{G}}

\newcommand{\PSL}{{\rm PSL}}

\newcommand{\GL}{{\rm GL}}

\newcommand{\Ind}{{\rm Ind}}
\usepackage{graphicx}
\usepackage{amsmath}
\usepackage{latexsym}
\usepackage{epsfig}
\usepackage{amsfonts}
\usepackage{amssymb}
\begin{document}
\title[Toroidal Fullerenes with the Cayley Graph Structures]{Toroidal Fullerenes with the Cayley Graph Structure}
\author{Ming-Hsuan Kang}
\address{Ming-Hsuan Kang\\ Department of Mathematics\\ The Pennsylvania State University\\
University Park, PA 16802} \email{\tt kang\_m@math.psu.edu}
\date{\today}
\thanks{The research is supported in part by the DARPA
grant HR0011-06-1-0012.
by the NSF grant DMS-0457574.}
\maketitle
\begin{abstract}
A central issue in molecular orbital theory is to compute the HOMO-LUMO gap of a molecule, which measures the excitability of the molecule. Thus it would be of interest to learn how to construct a molecule with the prescribed HOMO-LUMO gap. In this paper, we classify all possible structures of fullerene Cayley graphs and compute their spectrum. For any natural number $n$ not divisible by three, we show there exists an infinite family of fullerene graphs with the same HOMO-LUMO gap of size $\frac{2\pi}{\sqrt{3}n}+O(n^{-2})$. Finally, we discuss how to realize those families in three dimensional space.
\end{abstract}
\maketitle
\section{Introduction}
Since the discovery of the first fullerene, Buckministerfullerene $C_{60}$, fullerenes have attracted great interest in many scientific disciplines. Many properties of fullerenes can be studied using mathematical tools such as graph theory and group theory. Each fullerene can be represented by a 3-regular graph called the fullerene graph. Its vertices are carbon atoms of the molecule; two vertices are adjacent if there is a bond between corresponding atoms. One special feature of fullerene graphs is that they contain only pentagonal and hexagonal faces.\par
According to H\"{u}ckel molecular orbital theory, the energy spectrum of $\pi$-electrons of the fullerene can be approximated by eigenvalues of the adjacency matrix of the associated fullerene graph up to a constant multiple \cite{1}. One of the most important information of this energy spectrum is the HOMO-LUMO gap, which is the difference of energies between the highest occupied molecular orbit and the lowest unoccupied molecular orbit. Some partial results about the HOMO-LUMO gaps of certain families of graphs are known \cite{2,3,4,5}. However, it is in general difficult to construct a molecule with the prescribed HOMO-LUMO gap.
 \par
In this paper, we will consider a special kind of families of fullerene graphs: those with a Cayley graph structure. A Cayley graph $\Cayley(G,S)$ is a graph that encodes the structure of the group $G$ with a generating set $S$. It turns out that except for the case of $C_{60}$ which is a Cayley graph on $\PSL_2(\mathbb{F}_5)$ realized on the surface of a sphere, the remaining fullerene Cayley graphs are toroidal provided that they are orientable. There are different techniques to construct a family of toroidal fullerenes, some are based on combinatorial methods \cite{6,7} and some on geometric approaches \cite{8}. A great advantage of the Cayley graph structure is that all the eigenvalues can be explicitly  expressed using representations of the underlying group. \par
The plan of this paper is as follows. We first classify all possible structures of orientable fullerene Cayley graphs and compute their spectrum. For any natural number $n$ not divisible by 3, we show that there exists an infinite family of fullerene graphs with the same HOMO-LUMO gap of size $\frac{2\pi}{\sqrt{3}n}+O(n^{-2})$. Finally we discuss how to realize these families in three dimensional space.
\section{Classification of Toroidal Fullerene}
Let $G$ be a group generated by a finite set $S$. Assume that $S$ is symmetric, namely $S=S^{-1}$.
The Cayley graph $X=\Cayley(G,S)$ is defined as follows. Vertices of $X$ are elements in $G$ and two vertices $g_1, g_2 \in G$ are adjacent if $g_1= g_2 s$ for some $s\in S$.
The group $G$ acts on the graph $X$ by left multiplication. Observe that each vertex has exact $|S|$ neighbors and we call $X$ an $|S|$-regular graph.\par
$X=\Cayley(G,S)$ is called a fullerene graph if it is 3-regular
and it contains only pentagonal and hexagonal faces. In other words, a fullerene graph is a graph endowed with an extra structure: a set of faces enclosed by simple cycles of length five or six, and every edge is exactly shared by two faces. We shall assume that this extra structure is also preserved by the group action.
More precisely, if $\gamma=$ $(g,gs_1,\cdots$,$g s_1\cdots s_n\!=\!g)$ is the boundary of a face for some $s_i \in S$, then $g'\gamma$ is the boundary of another face for all elements $g' \in G$. In other word, it is enough to study the faces containing  the identity of $G$ and each of this face is represented by a relation $s_1\cdots s_n =id$.

Let $F$ be a face containing the identity with the boundary $(id,s_1,,s_1\cdots s_n\!=\!id)$. If we choose $s_1$ as the starting vertex, such that the boundary of $F$ is represented by $(s_1 , s_1s_2$,$\cdots$, $id$, $s_1)$. If we multiply $s_1^{-1}$ on the left to this boundary, we get $(id , s_2, \cdots$, $s_2\cdots s_n s_1\!=\!id)$ which represents the boundary of another face containing the identity. Thus, each cyclic permutation of $s_1$ $\cdots$ $s_n\!=\!id$ represents a different face containing the identity. When $X$ is a fullerene graph, which is 3-regular, there are three faces containing the identity and the cyclic permutations of $s_1 \cdots s_n$ can only contain up to three different elements. We conclude that $s_1 \cdots s_n$ must be equal to one of $s_1^5$, $s_1^6$, $(s_1 s_2)^3$, and $(s_1 s_2 s_3)^2$.

Write $S=\{a,b,c\}$ and regard $G$ as a quotient of a free group $F_3=\langle a,b,c\rangle$ by three kinds of relation: the orders of generators, the relations of faces and other relations. We shall first classify the group structures of $G$ without other relations and any finite fullerene graph is isomorphic to a finite quotient of $G$. We distinguish two cases:\\
Case (1): all generators have order two, say, $a^2=b^2=c^2=id$.\\
In this case, a relation of faces must have the form $(ab)^3=id$, $(abc)^2=id$, or their conjugations by permutating $a,b,c$. Note that $(ba)^3=id$ and $(ab)^3=id$ represents the same boundary of a face with opposite directions. Consequently, there are two subcases:\\
Case (1a): three relations of faces are  $(abc)^2=id, (bca)^2=id$ and $(cba)^2=id$.\\
Case (1b): three relations of faces are $(ab)^3=id, (bc)^3=id$ and $(ac)^3=id$.\\
\\
Case (2): at least one generator has order not equal to two, say, $a^2=id,b^2\neq id, c=b^{-1}$.
In this case, a relation of faces must have the form $b^5=id$, $b^6=id$, $(ab)^3=id$ or $(ab^2)^2=id$ . Therefore, there are three subcases:\\
Case (2a): three relations of faces are  $(ab^2)^2=id$, $(bab)^2=id$ and $(b^2a)^2=id$.\\
Case (2b): three relations of faces are  $b^6=id, (ab)^3=id$ and $(ba)^3=id$.\\
Case (2c): three relations of faces are  $b^5=id, (ab)^3=id$ and $(ba)^3=id$.\\
We conclude that
\begin{theorem}
A fullerene Cayley graph $\Cayley(G,S)$ is isomorphic to a finite quotient of one of the following five
graphs described in terms of generators and relations, and $S$ is the image of $S_i$ under the quotient.
\begin{itemize}
\item $G_1=\langle a,b,c| a^2=b^2=c^2=(ab)^3=(ac)^3=(bc)^3=id\rangle$;
\item $G_2=\langle a,b,c| a^2=b^2=c^2=(abc)^2=(bca)^2=(cab)^2=id\rangle $;
\item $G_3=\langle a,b| a^2=(ab^2)^2=(bab)^2=(b^2a)^2=id\rangle $;
\item $G_4=\langle a,b| a^2=b^6=(ab)^3=(ba)^3=id\rangle $;
\item $G_5=\langle a,b| a^2=b^5=(ab)^3=(ba)^3=id\rangle $.
\end{itemize}
\end{theorem}
It is not hard to prove that $G_5$ is isomorphic to $\rm{PSL}_2(\mathbb{F}_5)$ and in this case the graph is the well-known $C_{60}$.
For the first four cases, $X$ contains only hexagonal faces. Let $V$, $E$ and $F$ be the number of vertices, edges and faces of $X$, respectively. The euler characteristic of $X$ is equal to
$$V-E+F= |G|-\frac{3}{2}|G|+\frac{1}{2}|G|=0.$$
Therefore, $X$ lies on a torus if it is orientable; it lies on a Klein bottle if it is non-orientable. In order to embed $X$ in three dimensional space, we shall assume it is orientable. In the rest of paper, we shall consider finite toroidal fullerene Cayley graphs.
\section{Geometric Model of Toroidal Fullerene Cayley Graphs}
In this section, we will give each $G_i$ a geometric model. \par
Let $Y$ be the hexagonal tiling of the Euclidean plane $\R^2$ such that the origin $O$ is the center of a hexagon as shown in Figure 1.
\begin{center}
\includegraphics[width=0.2\textwidth]{full.1}\\
Figure 1:a hexagonal tiling of $\R^2$\\
\end{center}
For convenience, we use $\{\vec{e}_1,\vec{e}_2\}$ in Figure 1 as the basis of $\R^2$ to express linear transformations and translations. The group of affine transformations on $\R^2$ is the semi-direct product $W=\R^2 \rtimes \GL_2(\R) $. More precisely, the action of $(\vec{v},A) \in W$ on $\vec{u} \in \R^2$ is given by
$$ (\vec{v},A)(\vec{u})=\vec{v}+A\vec{u}, $$
and the group law is
$$ (\vec{v}_1,A_1)\cdot (\vec{v}_2,A_2)=(\vec{v}_1+A_1\vec{v}_2,A_1 A_2)$$
The group $W$ contains four different types of elements: rotations, reflections, translations, and glide reflections. Note that only rotations and reflections may have finite orders.
We shall construct an explicit embedding $\sigma_i$ of each group $G_i$ into $W$ and show that it induces a graph isomorphism from $\Cayley(G_i,S_i)$ to $Y$.\\
Let $\vec{e}_0=\vec{e}_1+\vec{e}_2$ and let $\rho: \rm{Aut}(Y)\rightarrow Y$ be the evaluation map given by $\rho(f)=f(\vec{e}_0)$. \\
Case (a): $G_1=\langle a,b,c| a^2=b^2=c^2=(abc)^2=id\rangle $.\\
Define $\sigma_1:G_1 \longrightarrow W$ by
\begin{eqnarray*}
&& \sigma_1(a)= \bigg(\left[\begin{matrix} 2 \\1\end{matrix}\right],\left[\begin{matrix} -1 & 0 \\ 0 & -1 \end{matrix}\right]\bigg) , \quad
\sigma_1(b)= \bigg(\left[\begin{matrix} 3 \\3\end{matrix}\right],\left[\begin{matrix} -1 & 0 \\ 0 & -1 \end{matrix}\right]\bigg)
, \\
&& \sigma_1(c)= \bigg(\left[\begin{matrix} 1 \\2\end{matrix}\right],\left[\begin{matrix} -1 & 0 \\ 0 & -1 \end{matrix}\right]\bigg).
\end{eqnarray*}
Here $\sigma_1(a), \sigma_1(b) $ and $ \sigma_1(c)$ are $180^\circ$-rotations centered at $a_0$, $b_0$, $c_0$, respectively,
as shown in Figure 2.
\\
\noindent Case (b): $G_2=\langle a,b,c| a^2=b^2=c^2=(ab)^3=(ac)^3=(bc)^3=id\rangle$.\\
Define $\sigma_2:G_2 \longrightarrow W$ by
\begin{eqnarray*}
&&\sigma_2(a)= \bigg(\left[\begin{matrix} 0 \\0\end{matrix}\right],\left[\begin{matrix} 1 & 0 \\ 1 & -1 \end{matrix}\right]\bigg) , \quad
\sigma_2(b)= \bigg(\left[\begin{matrix} 3 \\3\end{matrix}\right],\left[\begin{matrix} 0 & -1 \\ -1 & 0 \end{matrix}\right]\bigg)
, \\
&& \sigma_2(c)= \bigg(\left[\begin{matrix} 0 \\0\end{matrix}\right],\left[\begin{matrix} -1 & 1 \\ 0 & 1 \end{matrix}\right]\bigg).
\end{eqnarray*}
Here $\sigma_2(a), \sigma_2(b) $ and $ \sigma_2(c)$ are reflections with respect to the axes $a_0$, $b_0$, $c_0$, respectively  as shown in Figure 3.\\
\begin{minipage}{0.5\textwidth}
\begin{center}
\vspace{0.5cm}
\includegraphics[width=0.8\textwidth]{full.3}\\
\vspace{0.5cm}
Figure 2: $\rho \circ \sigma_1(G_1)$
\end{center}
\end{minipage}
\begin{minipage}{0.5\textwidth}
\begin{center}
\includegraphics[width=0.8\textwidth]{full.2}\\
Figure 3: $\rho \circ \sigma_2(G_2)$
\end{center}
\end{minipage}\\
\\
Case (c): $G_3= \langle a,b| a^2=(ab^2)^2=(bab)^2=(b^2a)^2=id\rangle $.\\
Define $\sigma_3:G_3 \longrightarrow W$ by
$$ \sigma_3(a)= \bigg(\left[\begin{matrix} 2 \\1\end{matrix}\right],\left[\begin{matrix} -1 & 0 \\ 0 & -1 \end{matrix}\right]\bigg) , \quad
\sigma_3(b)= \bigg(\left[\begin{matrix} 1 \\2\end{matrix}\right],\left[\begin{matrix} 1 & 0 \\ 1 & -1 \end{matrix}\right]\bigg).
$$
Here $\sigma_3(a)$ is a $180^\circ$-rotation around the point $a_0$ and $\sigma_3(b)$ is a glide reflection with respect to the axis $b_0$ as shown in Figure 4.
\\
\\
Case (d): $G_4=\langle a,b| a^2=b^6=(ab)^3=(ba)^3=id\rangle $.\\
Define $\sigma_4:G_4 \longrightarrow W$ by
$$ \sigma_4(a)= \bigg(\left[\begin{matrix} 2 \\1\end{matrix}\right],\left[\begin{matrix} -1 & 0 \\ 0 & -1 \end{matrix}\right]\bigg) , \quad
\sigma_4(b)= \bigg(\left[\begin{matrix} 2 \\1\end{matrix}\right],\left[\begin{matrix} 1 & -1 \\ 1 & 0 \end{matrix}\right]\bigg).
$$
Here $\sigma_4(a)$ is a $180^\circ$-rotation around the point $a_0$ and $\sigma_4(b)$ is a $60^\circ$-rotation around the point $b_0$ as shown in Figure 5.\\
\begin{minipage}{0.5\textwidth}
\begin{center}
\includegraphics[width=0.8\textwidth]{full.5}\\
Figure 4: $\rho \circ \sigma_3(G_3)$
\end{center}
\end{minipage}
\begin{minipage}{0.5\textwidth}
\begin{center}
\includegraphics[width=0.8\textwidth]{full.4}\\
Figure 5: $\rho \circ \sigma_4(G_4)$
\end{center}
\end{minipage}\\
\\
Observe that $\rho$ induces a graph isomorphism from $\Cayley\left(\sigma_i(G_i),\sigma_1(S_i)\right)$ to $Y$ for $i=1,2,3,4$.
To show that each $\sigma_i$ is an isomorphism , we use the following basic propositions in group theory.
\begin{proposition}\label{p1}
Let $H$ be a normal subgroup of $G$ and $\sigma$ be a homomorphism from $G$ to another group $G'$.
If $\sigma$ from $H$ to $\sigma(H)$ and  the induced map of $\sigma$ from $G/H$ to $\sigma(G)/\sigma(H)$ are both injective, then $\sigma$ is also injective.
\end{proposition}
\begin{proposition}\label{p2}
Any surjective homomorphism from $\Z^2$ to $\Z^2$ is an isomorphism.
\end{proposition}
Let $\vec{v}_1=2\vec{e}_1+\vec{e}_2$, $\vec{v}_2=\vec{e}_1+2\vec{e}_2$ and $T_{\vec{v}}$ denote the translation $\vec{x} \rightarrow \vec{x}+\vec{v}$.
\begin{theorem}
$\Cayley(G_i,S_i)$ is isomorphic to $Y$ as a graph for $i=1,2,3,4.$
\end{theorem}
\begin{proof}
Since $\Cayley\left(\sigma_i(G_i),\sigma_i(S_i)\right)$ is isomorphic to $Y$ as a graph,
we only need to show that $\sigma_i$ is injective and then it induces a graph isomorphism from
$\Cayley\left(G_i,S_i\right)$ to $\Cayley\left(\sigma_i(G_i),\sigma_i(S_i)\right)$.
We shall find, for each $G_i$, the translation subgroups $H_i$ and verify that $\sigma_i$ is injective on both $H_i$ and $G_i/H_i$ and hence $\sigma_i$ is injective by Proposition \ref{p1}.

Case (a) Let $H_1=\langle bc, ba \rangle$. It is easy to check that $H_1$ is an abelian normal subgroup of $G_1$.
Since $\sigma_1(H_1)=\left\langle \sigma_1(bc), \sigma_1(ba) \right\rangle=$
$\langle T_{\vec{v}_1}, T_{\vec{v}_2} \rangle \cong \Z^2$, we have that
$H_1$ is also isomorphic to $\Z^2$. By Proposition \ref{p2}, $\sigma_1$ is injective on $H_1$.
On the other hand, in $G_1/H_1$, $c=b^{-1}$, $a=b^{-1}$
and then $G_1/H_1=\langle b |b^2=id\rangle$. Since $b$ is not a translation, we have $\sigma(bH_1)\neq \sigma(H_1)$ and so $\sigma_1$ is injective on $G_1/H_1$.

Case (b) Let $H_2=\langle cbca, abac \rangle$. It is easy to check that $H_2$ is an abelian normal subgroup of $G_2$.
Since $\sigma_2(H_2)=\left\langle \sigma_2(abac), \sigma_2(cbca) \right\rangle=$
$\langle T_{2\vec{v}_1-\vec{v}_2}, T_{2\vec{v}_2-\vec{v}_1} \rangle \cong \Z^2$,
we have that $H_2$ is also isomorphic to $\Z^2$. By Proposition \ref{p2}, $\sigma_2$ is injective on $H_2$.
On the other hand, in $G_2/H_2$, $c=(aba)^{-1}$ and then $G_2/H_2 =\langle a,b| a^2=b^2=(ab)^3=id \rangle$.
Thus $G_2=\coprod gH_2$, where $g=id, a, b, ab, ba, aba$. Since none of $a, b, ab, ba, aba$ is a translation, $\sigma_2$ is injective on $G_2/H_2$.

Case (c) Let $H_3=\langle b^2, baba \rangle$. It is easy to check that $H_3$ is an abelian normal subgroup of $G_3$.
Since $\sigma_3(H_3)=\left\langle \sigma_3(b^2), \sigma_3(baba) \right\rangle=$
$\langle T_{\vec{v}_1}, T_{2\vec{v}_2-\vec{v}_1} \rangle \cong \Z^2$,  we have that
$H_3$ is also isomorphic to $\Z^2$. By Proposition \ref{p2}, $\sigma_3$ is injective on $H_3$.
On the other hand, $G_3/H_3 = \langle a,b| a^2=b^2=(ba)^2=id \rangle$.
Thus $G_3=\coprod gH_3$, where $g=id, a, b, ab$. Since none of $a, b, ab$ is a translation,  $\sigma_3$ is injective on $G_3/H_3$.

Case (d) Let $H_4=\langle b^3a, bab^2 \rangle$. It is easy to check that $H_4$ is an abelian normal subgroup of $G_4$.
Since $\sigma_4(H_4)=\left\langle \sigma_4(b^2), \sigma_4(baba) \right\rangle=$
$\langle T_{2\vec{v}_1-\vec{v}_2}, T_{2\vec{v}_2-\vec{v}_1} \rangle \cong \Z^2$,
 we have that $H_4$ is also isomorphic to $\Z^2$. By Proposition \ref{p2}, $\sigma_4$ is injective on $H_4$.
On the other hand, in $G_4/H_4$, $a=b^{-3}$ and then $G_4/H_4 = \langle b|b^6=id \rangle$.
Thus $G_4=\coprod b^{i}H_4$, where $0 \leq i \leq 5$. Since none of $b^i$ is a translation for $1\leq i \leq 5$,
, $\sigma_4$ is injective on $G_4/H_4$.
\end{proof}

\emph{For now on, we identify each $G_i$ with $\sigma_i(G_i)$ as a subgroup of $W$.}
\section{Finite Toroidal Fullerenes}
The Cayley graphs discussed in the previous section are infinite graphs
and they are all isomorphic to the hexagonal tiling $Y$ as graphs.
To obtain finite fullerene graphs, we shall consider finite quotients of these graphs.
Recall that a torus is obtained as the quotient of the Euclidean plane by a lattice.
In our case,  the translation subgroup of $W$ plays the same role as lattices.
Observe that the translation group $T$ preserving $Y$ is spanned by translations $T_{\vec{v}_1}$ and $T_{\vec{v}_2}$.
Thus every finite toroidal fullerene comes from a quotient of $\Cayley(G_i,S_i)$ by a normal subgroup $H$ of $G_i$ contained in $T$. Observe that $T \subset G_1$ and every subgroup of $T$ is normal in $G_1$.
Consequently, all finite toroidal fullerenes arise from quotients of $\Cayley(G_1,S_1)$ by its translation subgroups.
Let $ s= \left[\begin{matrix} -1 & 0 \\ 0 & -1 \end{matrix}\right]$, then we can write the generators $a,b$ and $c$ of $G_1$ as
$$ a=(\vec{v}_1,s), \quad b=(\vec{v}_1+\vec{v}_2,s),\quad c=(\vec{v}_2,s),\quad \rm{and}\quad G_1= \langle \vec{v}_1,\vec{v}_2 \rangle \rtimes \langle s \rangle  \cong \Z^2 \rtimes \Z/2\Z $$
We conclude that
\begin{theorem}
Every finite toroidal fullerene is isomorphic to some $X_N=\Cayley(G_N,\{a,b,c\})$, where
$N$ is a rank two subgroup of $\Z^2$, $G_N=\Z^2/N \rtimes \Z/2\Z$ and
$a=(\vec{v}_1,s)$, $b=(\vec{v}_1+\vec{v}_2,s)$ and $c=(\vec{v}_2,s)$.
\end{theorem}
\section{Spectrum of A Finite Fullerene}
In this section, we compute the spectrum of the toroidal fullerene graph $X_N$. Recall that if $\Phi$ is the right regular representation of $G_N$, then the adjacent matrix of $X_N$ is equal to $\Phi(a)+\Phi(b)+\Phi(c)$. On the other hand, $\Phi$ decomposes as the sum of all irreducible representations $\{\Phi_i\}$ of $G_N$ with multiplicity equal to the dimension of $\Phi_i$. Thus, it suffices to compute the eigenvalues of $\Phi_i(a)+\Phi_i(b)+\Phi_i(c)$. Denote by $\hat{A}$  the set of characters of an abelian group $A$. To find all irreducible representations of $G_N$, we apply the following theorem \cite{9}:
\begin{theorem}
Let $G= A \rtimes H$, where $H$ is a subgroup and $A$ is a normal abelian subgroup of $G$.
For $\chi \in \hat{A}$, let
$H_\chi=\{ h \in H | \chi(h^{-1}a h)=\chi(a), \forall a\in A\}$. Let $\rho$ be an irreducible representation of $H_\chi$. Extend $\chi$ and $\rho$ to a representation of $A H_\chi$. Then $\Ind^{G}_{A H_\chi}(\chi \otimes \rho)$ is an irreducible representation of G and all irreducible representations of $G$ come from this construction.
\end{theorem}
In our case, $G=G_N$, $H= \mathbb{Z}/2\mathbb{Z}=\langle s\rangle $,
$A= \mathbb Z^2/N$ and $H_\chi= \{id\}$ or $\langle s\rangle $. If $H_{\chi}=\{id\}$, then it has only the trivial representation. In this case, $\Ind^{G}_{A H_\chi}(\chi \otimes \rho)={\rm Ind}^G_A(\chi)$ is two dimensional such that
\begin{eqnarray*}
{\rm Ind}^G_A(\chi)(a)&=&\left(\begin{matrix} 0 & \overline{\chi(\vec{v}_1)} \\ \chi(\vec{v}_1)& 0 \end{matrix}\right),
{\rm Ind}^G_A(\chi)(b)=\left(\begin{matrix} 0 & \overline{\chi(\vec{v}_1)\chi(\vec{v}_2)} \\ \chi(\vec{v}_1)\chi(\vec{v}_2) & 0 \end{matrix}\right)\\
{\rm Ind}^G_A(\chi)(c)&=&\left(\begin{matrix} 0 & \overline{\chi(\vec{v}_2)} \\ \chi(\vec{v}_2) & 0 \end{matrix}\right).
\end{eqnarray*}
So \\
$
{\rm Ind}^G_A(\chi)(a)+ {\rm Ind}^G_A(\chi)(b)+ {\rm Ind}^G_A(\chi)(c)= \\
\left(\begin{matrix} 0 & \overline{\chi(\vec{v}_1)}+\overline{\chi(\vec{v}_1)\chi(\vec{v}_2)}+\overline{\chi(\vec{v}_2)} \\ \chi(\vec{v}_1)+\chi(\vec{v}_1)\chi(\vec{v}_2)+\chi(\vec{v}_2) & 0 \end{matrix}\right),$\\

which has eigenvalues $\pm |  \chi(\vec{v}_1)+ \chi(\vec{v}_2)+\chi(\vec{v}_1)\chi(\vec{v}_2)|$. \\
If $H_{\chi}=\langle s\rangle $ , then $\Ind^{G}_{A H_\chi}(\chi \otimes \rho)=\chi \rho$. Since $\rho \in \hat{H}$ takes values in $\{1, -1\}$ and eigenvalues of the adjacency matrix are all real, we have
\begin{eqnarray*}
\rho \chi(a)+\rho \chi(b)+\rho \chi(c)&=&\rho(s)\bigg(\chi(\vec{v}_1)+\chi(\vec{v}_1+\vec{v}_2)+\chi(\vec{v}_2)\bigg)\\
&=&|\chi(\vec{v}_1)+\chi(\vec{v}_1+\vec{v}_2)+\chi(\vec{v}_2)| \mbox{ or } -|\chi(\vec{v}_1)+\chi(\vec{v}_1+\vec{v}_2)+\chi(\vec{v}_2)|.
\end{eqnarray*}
We conclude that
\begin{theorem}
The spectrum of $X_N$ is
$\bigg\{\pm |\chi(\vec{v}_1)+ \chi(\vec{v}_2)+\chi(\vec{v}_1)\chi(\vec{v}_2)|\,\bigg| \chi \in \hat{A} \bigg\}.$
\end{theorem}
Recall that if the cardinality of the graph is equal to $M$, the HOMO eigenvalue $\lambda_H$ is the $(\frac{M}{2}+1)$-th largest eigenvalue, the LUMO eigenvalue $\lambda_L$ is the $(\frac{M}{2})$-th largest eigenvalue and $\lambda_H-\lambda_L$ is called the HOMO-LUMO gap\footnote{In terms of eigenvalues of the laplacian, $3-\lambda_H$ corresponds to HOMO and $3-\lambda_L$ corresponds to LUMO.}.
For bipartite graphs, $\lambda_H$ is equal to the smallest nonnegative eigenvalue; $\lambda_L$ is equal to the largest non-positive eigenvalue and $\lambda_L=-\lambda_H$. Note that $G_N=A \coprod As$ gives a bipartite structure of $X_N$ and so we have
\begin{corollary} \label{c1}
$\lambda_H-\lambda_L=2 \min_{\chi \in \hat{A}}|\chi(\vec{v}_1)+ \chi(\vec{v}_2)+\chi(\vec{v}_1)\chi(\vec{v}_2)|.$
\end{corollary}
A character $\chi$ of $\Z^2$ is uniquely determined by $\chi(\vec{v}_1)= \exp(i\theta_1)$ and
$\chi(\vec{v}_2)= \exp(i\theta_2)$, where $\theta_1, \theta_2 \in (\R/2\pi\Z)^2$. We identify $\chi \in \hat{\Z}^2 $ with $(\theta_1,\theta_2) \in (\R/2\pi\Z)^2.$ The characters of $\Z^2/N$ are the characters of $\Z^2$ trivial on $N$, so it can be identified with
$$N^{\perp}=\left\{(\theta_1,\theta_2) \in (\R/2\pi\Z)^2 \bigg|
\exp\left(i (\theta_1 x+ \theta_2 y)\right)=1, \mbox{for all }(x,y)\in N\right\}.$$
Consider a function $f$ defined by
\begin{eqnarray*}
f(\theta_1,\theta_2)&=&| \exp(i\theta_1)+ \exp(i\theta_2)+\exp(i\theta_1+i\theta_2)|^2 \\
&=&3+2\cos \theta_1+2\cos \theta_2+2\cos \theta_1 \cos \theta_2+2\sin \theta_1\sin \theta_2.
\end{eqnarray*}
We can reformulate Corollary \ref{c1} as
\begin{corollary} \label{c2}
$$ \lambda_H-\lambda_L = \min_{(\theta_1,\theta_2)\in N^\perp} 2\sqrt{f(\theta_1,\theta_2)}. $$
\end{corollary}
This function $f$ satisfies the following proposition used in the next section.
\begin{proposition}\label{lemma1}
For $-\pi \leq \theta \leq \pi$, $f(\theta,x)\geq f(\theta,\frac{\theta}{2}+\pi)$ for all $x$.
\end{proposition}
\begin{proof}
Observe that when $-\pi \leq \theta \leq \pi$, $\cos\frac{\theta}{2}$ is always nonnegative and then
\begin{eqnarray*}
f(\theta,x)&=& 3+2\cos \theta +2\cos x(1+\cos \theta)+2\sin \theta\sin x \\
&=& 3+2\cos \theta+ 2(1+\cos\theta,\sin\theta)\cdot (\cos x,\sin x) \\
&=& 3+2\cos \theta+ 4\cos\frac{\theta}{2}(\cos\frac{\theta}{2},\sin\frac{\theta}{2})\cdot (\cos x,\sin x)\\
&\geq& 3+2\cos \theta- 4\cos\frac{\theta}{2},
\end{eqnarray*}
and the equality holds when $$(\cos x, \sin x) =(-\cos\frac{\theta}{2},-\sin\frac{\theta}{2})=\left(\cos(\pi+\frac{\theta}{2}),\sin(\pi+\frac{\theta}{2})\right).$$
\end{proof}

\section{Families of Toroidal Fullerenes With Given HOMO-LUMO Gap}
In this section, we will show that for every natural number $p$ not divisible by 3, there is an infinite family $\{X_{p,q}\}_q$ with the same HOMO-LUMO gap of size $\frac{2\pi}{\sqrt{3}p}+O(p^{-2})$.\par
Let $N_{p,q}$ be the sublattice of $\Z^2$ generated by $(p,0)$ and $(-q,2q)$ and $ X_{p,q}=\Cayley(\Z^2/N_{p,q} \rtimes \Z/2\Z ,\{a,b,c\})$. Observe that $N^\perp$ is generated by $\left\{(\frac{2\pi}{p},\frac{\pi}{p}), (0,\frac{\pi}{q})\right\}$. By Proposition \ref{lemma1}, we have
\begin{eqnarray*}
\lambda_H(X_{p,q})-\lambda_L(X_{p,q})
&=&  \min_{(\theta_1,\theta_2)\in N^\perp} 2\sqrt{f(\theta_1,\theta_2)}\\
&=&  \min_{u,v} 2\sqrt{f\bigg(\frac{2\pi u}{p},\frac{\pi u}{p}+\frac{\pi v}{q}\bigg)}\\
&=&  \min_{u} 2\sqrt{f\bigg(\frac{2\pi u}{p},\frac{\pi u}{p}+\pi\bigg)}
\end{eqnarray*}
For convenience, define $g(\theta)=f(\theta,\frac{\theta}{2}+\pi)=3+2\cos \theta- 4\cos\frac{\theta}{2}$ for $-\pi\leq \theta \leq \pi$.
Since $g(-\theta)=g(\theta)$ , we have
$$ \lambda_H(X_{p,q})-\lambda_L(X_{p,q})=\min_{u} 2\sqrt{g\bigg(\frac{2\pi u}{p}\bigg)}
=\min_{0\leq \frac{2\pi u}{p} \leq \pi} 2\sqrt{g\bigg(\frac{2\pi u}{p}\bigg)} $$
It is easy to check that $g(\theta)$ decreases in $[0,\frac{2 \pi}{3}]$ and $g(\theta)$ increases in  $[\frac{2\pi}{3},\pi]$ , so the minimum of $g(\frac{2\pi u}{p})$ occurs when $u=\lfloor\frac{p}{3}\rfloor$ or $\lceil\frac{p}{3}\rceil$, where
$\lfloor x \rfloor$ is the floor function and $\lceil x \rceil$ is the ceiling function.
On the other hand, the Taylor series expansion of $2\sqrt{g(\theta)}$ at $\theta=\frac{2\pi}{3}$ is given by
$$ 2\sqrt{g(\theta)} = \sqrt{3}\left|\theta-\frac{2\pi}{3}\right| + O\left( \left|\theta-\frac{2\pi}{3}\right|^2\right). $$
If $3\nmid p$, then $\min\left\{\left|\frac{2\pi}{p}\lfloor\frac{p}{3}\rfloor-\frac{2\pi}{3}\right|,
\left|\frac{2\pi}{p}\lceil\frac{p}{3}\rceil-\frac{2\pi}{3}\right|\right\}=\frac{2\pi}{3p}$. We conclude that
\begin{theorem}
If $3\nmid p$, then $\lambda_L(X_{p,q})-\lambda_H(X_{p,q})=2 \min\left\{ \sqrt{g\left(\frac{2\pi}{p}\lfloor\frac{p}{3}\rfloor\right)}, \sqrt{g\left(\frac{2\pi}{p}\lceil\frac{p}{3}\rceil\right)}\right\}= \frac{2\pi}{\sqrt{3}p}+O(p^{-2}),$ which is independent of $q$.
\end{theorem}
\section{Imbedding Graph Into Three Dimensional Space}
For the purpose of real world applications, it is important to design fullerenes which have a good imbedding in three dimensional space with small potential energy. The first constraint of such imbedding is that the length of each edge should be almost equal. This constraint will be satisfied if the imbedding of fullerene graphs is isometric.
Note that for the lattices $N'_{p,q}=N_{p,q}(\vec{e}_0)$ on the Euclidean plane $\R^2$ has a fundamental domain $\Sigma$ spanned by the orthogonal basis $p\vec{v}_1$ and $-q\vec{v}_1+2q \vec{v}_2$. Choose the length of $\vec{v}_1$ as the unit of length. The width of $\Sigma$ equals $p$; its length equals $2 \sqrt{3}q$.
The fullerene graphs $X_{p,q}$ is obtained from the hexagon tiling $Y$ in $\R^2$ quotient by the lattices $N'_{p,q}$, so it suffices to find a map from $\R^2 /N'_{p,q}$ to three dimensional space, which induces an imbedding of $X_{p,q}$ to three dimensional space. Consider the standard map $F(u,v)=(x(u,v),y(u,v),z(u,v))$ from $\Sigma$ to a torus, where\\
$\left\{\begin{array}{rcl}
x(u,v)&=&\left(R+r\cos(\frac{2 \pi u}{p})\right)\cos(\frac{\pi v}{ \sqrt{3}q}),\\
y(u,v)&=&\left(R+r\cos(\frac{2 \pi u}{p})\right)\sin(\frac{\pi v}{\sqrt{3}q}),\\
z(u,v)&=&r\sin(\frac{2 \pi u}{p}),
\end{array}\right.$ \\
and $R$ and $r$ are some constants.
In order that $F$ is an isometry, it should satisfy $\langle F_u,F_v\rangle=0$, $\langle F_u,F_u\rangle = \frac{4\pi^2 r^2}{p^2} =1$ and $\langle F_v,F_v\rangle=\frac{\pi^2 R^2}{3q^2}\left|1+\frac{r}{R}\cos(\frac{2 \pi u}{p})\right|^2=1$, which is impossible. However, we can simply let $r=\frac{p}{2 \pi}$ and $R=\frac{\sqrt{3}q}{ \pi}$, such that $\langle F_u,F_u\rangle=1$ and $\langle F_v,F_v\rangle=\left|1+\frac{p}{2\sqrt{3}q}\cos(\frac{2 \pi u}{p})\right|^2$. Then $F$ is close to an isometry when $\frac{p}{q}$ is small. Thus with $p$ fixed and $q$ increasing , we can increase the stability of the molecule without changing the HOMO-LUMO gap. Note that when $\frac{p}{q}\to 0$,  the toroidal fullerene becomes a carbon nanotube.\\
\begin{minipage}{0.5\textwidth}
\begin{center}
\includegraphics[width=\textwidth]{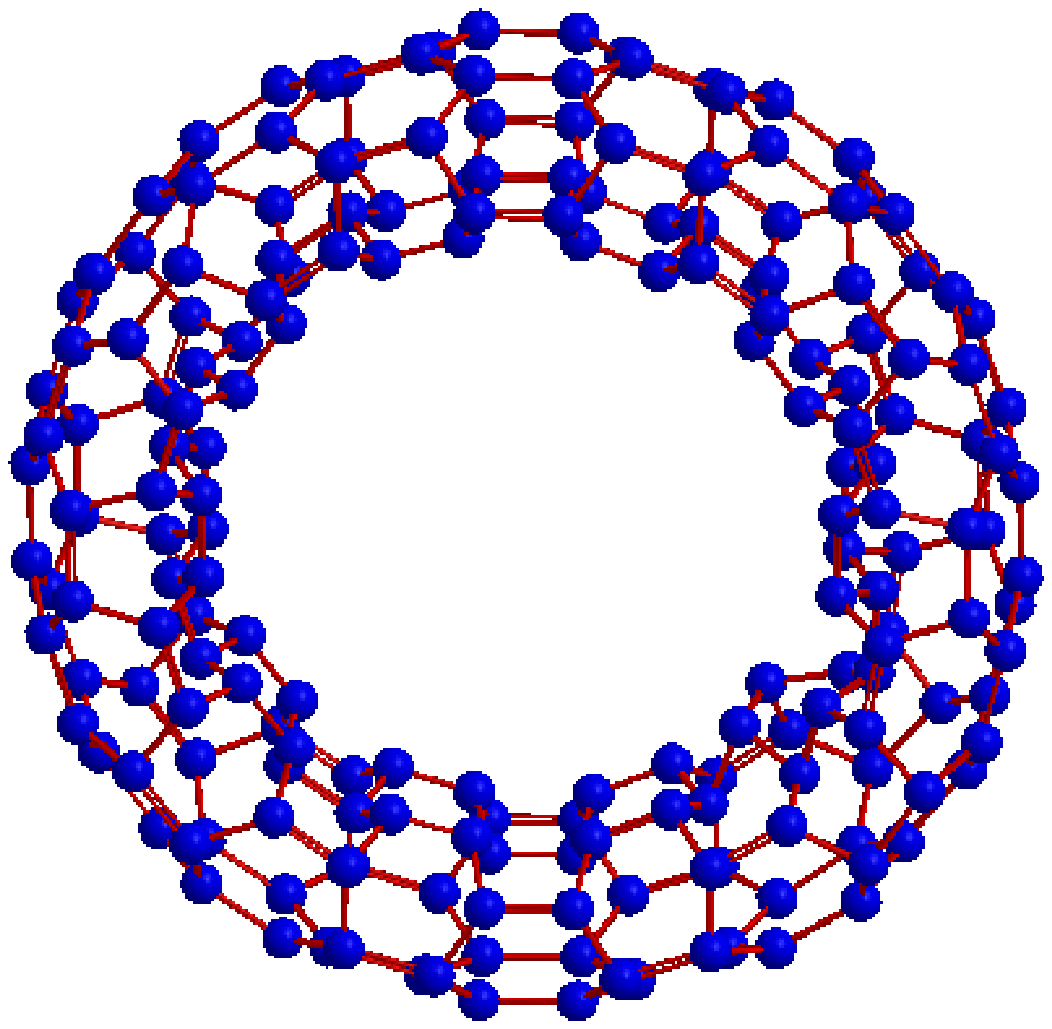}\\
Figure 8: $p=5,q=10$
\end{center}
\end{minipage}
\begin{minipage}{0.5\textwidth}
\begin{center}
\includegraphics[width=\textwidth]{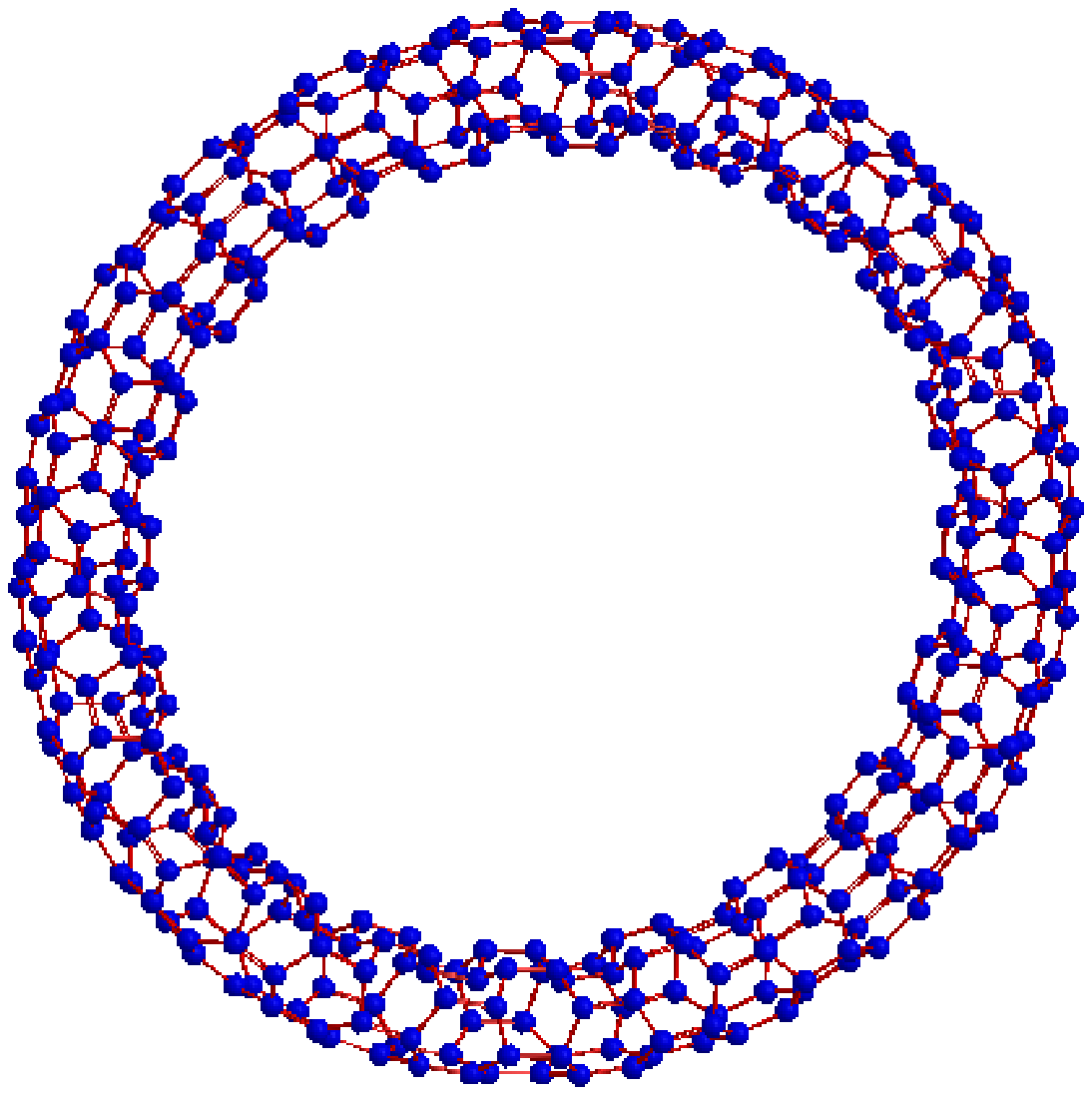}\\
Figure 9: $p=5,q=20$
\end{center}
\end{minipage}
\section{Acknowledgement}
The author gratefully acknowledges valuable discussions with Professor Wen-Ching Winnie Li.

\end{document}